\documentclass[10pt, a4paper]{amsart}

\usepackage{amsmath,amssymb,enumerate}

\usepackage[T1]{fontenc}

\usepackage{babel}
\usepackage{amstext}
\usepackage{amsmath}
\usepackage{amsfonts}
\usepackage{latexsym}
\usepackage{ifthen}
\usepackage{amssymb}
\usepackage{array}
\usepackage{makecell}
\usepackage{mathrsfs}

\newcommand{\mm}{\mathfrak m}
\newcommand{\wt}{\widetilde}

\DeclareMathOperator{\Tr}{Tr}

\newcommand{\codim}{{\rm codim}}

\newtheorem{lemma1}{}[section]

\newenvironment{lemma}{\begin{lemma1}{\bf Lemma.}}{\end{lemma1}}
\newenvironment{example}{\begin{lemma1}{\bf Example.}\rm}{\end{lemma1}}

\newenvironment{theorem}{\begin{lemma1}{\bf Theorem.}}{\end{lemma1}}

\newenvironment{proposition}{\begin{lemma1}{\bf Proposition.}}{\end{lemma1}}
\newenvironment{corollary}{\begin{lemma1}{\bf Corollary.}}{\end{lemma1}}
\newenvironment{remark}{\begin{lemma1}{\bf Remark.}\rm}{\end{lemma1}}

\newenvironment{definition}{\begin{lemma1}{\bf Definition.}}{\end{lemma1}}

\newenvironment{conjecture}{\begin {lemma1}{\bf Conjecture.}}{\end{lemma1}}

\newenvironment{remark*}{{\bf Remark.}}{}
\newenvironment{remarks*}{{\bf Remarks.}}{}
\newenvironment{example*}{{\bf Example.}}{}
\newenvironment{assumption*}{{\bf Assumption.}}{}

\newcommand{\Q}{\ensuremath{\mathbb{Q}}}
\newcommand{\Z}{\ensuremath{\mathbb{Z}}}
\newcommand{\C}{\ensuremath{\mathbb{C}}}
\newcommand{\N}{\ensuremath{\mathbb{N}}}
\newcommand{\PP}{\ensuremath{\mathbb{P}}}

\usepackage{eurosym}

\newcommand{\holom}[3]{\ensuremath{#1\colon#2  \rightarrow #3}}
\newcommand{\fibre}[2]{\ensuremath{#1^{-1} (#2)}}

\makeatletter
\ifnum\@ptsize=0  \fi \ifnum\@ptsize=2  \fi 

\newcommand\sO{{\mathcal O}}

\DeclareMathOperator*{\sing}{sing}

\usepackage{xcolor}

\usepackage{hyperref}
\hypersetup{
	colorlinks=true,
	linkcolor=red,
	citecolor=blue}

\makeatletter
\@namedef{subjclassname@2020}{\textup{2020} Mathematics Subject Classification}
\makeatother


\newcommand{\cO}{\mathcal{O}}

\newcommand{\CC}{\mathbb{C}}
\newcommand{\QQ}{\mathbb{Q}}

\address{Ama\"el Broustet, Universit{\'e} de Lille, CNRS, UMR 8524 Laboratoire Paul Painlevé, Département de math{\'e}matiques, F-59000 Lille, France.}
\email{amael.broustet@univ-lille.fr}

\address{Andreas H\"oring, Universit\'e C\^ote d'Azur, CNRS, LJAD, France}
\email{Andreas.Hoering@univ-cotedazur.fr}

\title{Singularities of totally invariant hypersurfaces of endomorphisms}
\author{Amaël Broustet \and Andreas Höring}
\date{September 8th, 2026} 

\begin{document}

\begin{abstract} 
Let $f\colon X\to X$ be a polarised endomorphism of a  complex
projective manifold, and let $D\subset X$ be a reduced totally invariant
hypersurface.  We prove that every irreducible component of the singular locus of $D$
has codimension one in $D$. In particular, if $D$ is normal, then it is
smooth.  As an application we prove the linearity of totally
invariant divisors for endomorphisms of $\PP^4$.
\end{abstract}

\maketitle

\section{Introduction} 

\subsection{Main result}

Let $X$ be a projective manifold of dimension $n\ge 2$. An endomorphism is a non-invertible finite morphism $f \colon X\to X$.
One says that the endomorphism $f$
%$f\colon X\to X$ 
is polarised if there exist an ample line bundle $L$ on $X$ and an integer $q>1$ such that
$ f^*L \simeq L^{\otimes q}.$
A reduced divisor $D\subset X$ is said to be totally invariant if we have a set-theoretical equality
$$
f^{-1}(D)=D.
$$ 
Note that this implies that $D$ itself admits a polarised endomorphism $\holom{f|_D}{D}{D}$, but the property of being totally invariant is typically much stronger.
A classical example of this situation is the endomorphism
$$
\holom{f}{\PP^n}{\PP^n}, \qquad [X_0:\ldots:X_n] \ \mapsto \ [X_0^q:\ldots:X_n^q] 
$$
for which each coordinate hyperplane $\{ X_i=0 \}$ is totally invariant.
The importance of totally invariants subsets is due to their role for the unique maximal entropy measure of $f$, as shown by the work of Forn\ae ss-Sibony \cite{FS94} and Briend-Duval \cite{BD01}.
More recently totally invariant subvarieties and their singularities have been investigated by Zhong \cite{Zho21},  Yoshikawa \cite{Yos22}, Chang-Luo \cite{CL26} and Luo-Meng \cite{LM26}. Their work connects them
to the large body of the work on the existence of log Calabi-Yau structures on varieties admitting a polarised
endomorphism \cite{Men23, CZ25, MYY25}, as conjectured by the first-named author and Gongyo \cite{BG17}. 

Totally invariant hypersurfaces are expected to be very special:
\begin{conjecture} \label{conj:linear}
Let $\holom{f}{\PP^n}{\PP^n}$ be an endomorphism of the projective space.
If $D \subset \PP^n$ is a totally invariant prime divisor, then $D$ is a hyperplane.
\end{conjecture}

If $D$ is smooth this follows from results of Paranjape and Srinivas \cite{ParanjapeSrinivas1989} and Cerveau and Lins Neto \cite{CerveauLinsNeto2000}. Beauville  \cite{Beauville2001} showed the stronger property that smooth hypersurfaces of dimension at least two and degree at least three do not admit an endomorphism.
Beauville's statement is not true for singular hypersurfaces of the projective space as shown by an example of Zhang \cite[Ex.1.9]{Zha14}.
A more refined approach taking into account the total invariance allows to exclude that
in the situation of 
Conjecture \ref{conj:linear} 
the divisor $D$ has degree $n+1$
\cite[Theorem~2.1]{HN11} or $n$
\cite[Theorem~1.1]{Hoe17}. In particular Conjecture \ref{conj:linear} is known if $n \leq 3$.

Our main result is the following.

\begin{theorem}\label{thm:main}
Let $X$ be a projective manifold, and let
$f\colon X\to X$ be a polarised endomorphism.
Let $D \subset X$ be a reduced hypersurface that is totally invariant.
Then $D$ is smooth or not normal.
\end{theorem}

As an immediate consequence we obtain Conjecture \ref{conj:linear} in the normal case:

\begin{corollary}\label{cor:intro-pn}
For $n\geq2$, let $f\colon \PP^n \to \PP^n$ be an endomorphism, and let
$D \subset \PP^n$ be a totally invariant hypersurface that is normal. Then $D$ is a hyperplane.
\end{corollary}

Combining our technique with classification results we also obtain:

\begin{corollary}\label{cor:cubic}
Let $f\colon\PP^n\to\PP^n$ be an endomorphism, and let $D\subset\PP^n$ be a totally invariant prime
divisor. Then $\deg D\neq3$.
\end{corollary}

As a consequence we prove
Conjecture~\ref{conj:linear} in dimension four.

\begin{corollary}\label{cor:linear-p4}
Let $f\colon\PP^4\to\PP^4$ be an endomorphism,
and let $D\subset\PP^4$ be a totally invariant prime divisor. Then $D$ is
a hyperplane.
\end{corollary}

\subsection{Strategy of proof}

The starting point is our paper 
\cite{BroustetHoring2014} where we showed that in the situation of Theorem \ref{thm:main} the pair $(X,D)$ is log-canonical. The proof relies on general techniques from the MMP for pairs introduced in this context by Nakayama \cite{Nak23} and does not really use that the ambient space $X$ is smooth.
By contrast, we show in this paper that when $X$ is smooth and $D$ is singular, then $D$ does not have rational singularities. In particular the pair $(X,D)$ is not purely log terminal. 
The key idea is to replace general invariants like the log-canonical threshold with the finer information provided by the roots of the Bernstein-Sato polynomial of a hypersurface
in a manifold (cf. Definition \ref{defn:bs-global}).

With this idea in mind we use a number of results from singularity theory:
a theorem of Giraldo and Roeder \cite{GiraldoRoeder2020} guarantees that the singular locus $D_{\sing}$ is also totally invariant. 
By a theorem of Fakhruddin \cite{Fakhruddin2003} the periodic points
of $f|_{D_{\sing}}$ are dense in $D_{\sing}$, so (up to taking an iterate of $f$) we can choose a point $x \in D_{\sing}$ such that $f(x)=x$.
Now a root-theoretic argument using in a crucial way that $D$ is totally invariant shows
that the Bernstein-Sato polynomial in every point $x \in D$ is equal $(s+1)^{d_x}$ for some $d_x \in \N$.
Classical results on the Steenbrink spectrum and a theorem of A'Campo \cite{ACampo1973} 
then yield the main result. 

In the proof of Corollary~\ref{cor:cubic}, the main step is to show by contradiction that a non-normal cubic threefold $D \subset \PP^4$
cannot be $f$-totally invariant. 
A theorem of Lee, Park and Schenzel \cite{LPS11} shows that $D$ is a cone or a special cubic threefold $F_5$. A result of Mabed \cite{Mabed2023} excludes the first case,
in the second case the non-normal locus of $D$ is a linear space $\PP^2$
that is also $f$-totally invariant. We then show that the induced endomorphism
$\holom{f|_{D_{\sing}}}{\PP^2}{\PP^2}$ leaves a smooth conic totally invariant,
a contradiction to \cite{CerveauLinsNeto2000}.

{\bf Acknowledgements.} We thank Alex Dimca for helpful communications.

The first author acknowledges the support of the CDP C2EMPI, together with the French State under the France-2030 programme, the University of Lille, the Initiative of Excellence of the University of Lille, the European Metropolis of Lille for their funding, and support of the Labex CEMPI (ANR-11-LABX-0007-01) and the R-CDP-24-004-C2EMPI project.

The second-named author was supported  by the France 2030 investment plan managed by the ANR, as part of the Initiative of Excellence of Universit\'e C\^ote d'Azur, reference ANR-15-IDEX-01. 
He was also supported by the ANR-DFG project ``Positivity on K-trivial varieties'', ANR-23-CE40-0026 and DFG Project-ID 530132094,
and the ANR-Project ``Groups in algebraic geometry'' ANR-24-CE40-3526.

\section{Notations and basic results}

We work over the complex field $\C$, for general definitions we refer to Hartshorne's book \cite{Har77}.
We will use standard terminology and results 
of the minimal model program (MMP) as explained in \cite{KM98}. 
A variety is an integral scheme of finite type over $\C$, complex manifolds are always supposed to be connected.

\begin{lemma}\label{lem:m=q}
Let $X$ be a projective manifold of dimension $n$, and let $f\colon X\to X$ be a polarised endomorphism with $f^*L\simeq L^{\otimes q}$ for an ample line bundle $L$ and $q>1$.
Let $D\subset X$ be a reduced divisor such that each irreducible component is totally invariant.
Then $f^*D = q D$.
\end{lemma}

\begin{proof}
On the one hand the morphism $f$ is finite, so  $(f^*L)^n=(\deg f)L^n$. 
On the other hand one has
\[
 (f^*L)^n = (L^{\otimes q})^n = q^n L^n.
\]
Since $L^n>0$ for an ample divisor we obtain $\deg f=q^n$.
Thus we have  $f_*\bigl(f^*(L^{n-1})\bigr)= q^n L^{n-1}$ by \cite[Lemma 1]{Beauville2001}.

{\em Assume that $D$ is a prime divisor.}
Since $f^* D$ has support on $D$ we have $f^*D=m D$ for some $m \in \N$. 
Intersecting with $L^{n-1}$ and using the projection formula, we obtain
\[
 m(D\cdot L^{n-1})=(f^*D\cdot L^{n-1})=(D\cdot f_*L^{n-1})
=  (D\cdot f_* f^*(\frac{1}{q^{n-1}} L^{n-1}))
 =q(D\cdot L^{n-1}).
\]
Since $D$ is effective and $L$ is ample, we have $(D\cdot L^{n-1})>0$. Therefore $m=q$.

{\em Conclusion.} If $D=\sum D_i$ is the decomposition into irreducible components 
we have $f^* D_i = m_i D_i$ since each component is totally invariant by assumption.
Yet we have just shown that $m_i=q$ is independent of the component.
\end{proof}

\begin{lemma}\label{lem:restrictpolarised}
Let $X$ be a projective manifold of dimension $n$, and let $f\colon X\to X$ be a polarised endomorphism with $f^*L\simeq L^{\otimes q}$ for an ample line bundle $L$ and $q>1$.
Let $Z\subset X$ be a subvariety such that $f(Z)=Z$.
Then the restriction $f|_Z\colon Z\to Z$ is a polarised endomorphism.
\end{lemma}

\begin{proof}
We have
\[
 (f|_Z)^*(L|_Z)\simeq (f^*L)|_Z\simeq (L^{\otimes q})|_Z\simeq (L|_Z)^{\otimes q},
\]
and $L|_Z$ is ample. Thus $f|_Z$ is polarised.
\end{proof}

We start with an immediate consequence of a theorem of Giraldo and Roeder \cite{GiraldoRoeder2020}:

\begin{lemma}\label{lem:backward}
Let $X$ be a complex manifold, and let $f\colon X\to X$ be a finite surjective morphism.
Let $D\subset X$ be a reduced hypersurface that is totally invariant for $f$.
Then
\[
 f^{-1}(D_{\sing})\subseteq D_{\sing}.
\]
\end{lemma}

\begin{proof}
Let $x\in D_{\sing}$ and let $y\in f^{-1}(x)$.
Choose analytic neighborhoods of $x$ and $y$ and let $h$ be a reduced local equation of $D$ at $x$.
The germ $(D,x)$ is singular, and its inverse image under the finite germ of $f$ at $y$ is precisely the germ $(D,y)$ because $f^{-1}(D)=D$ set-theoretically.
By \cite[Thm.A]{GiraldoRoeder2020} the germ $(D,y)$ is singular.
\end{proof}

\begin{proposition}\label{prop:periodic-component}
Let $X$ be a projective manifold, and let $f\colon X\to X$ be a finite surjective morphism.
Let $D\subset X$ be a reduced hypersurface that is totally invariant for $f$.

Then there exists an iterate $f^N$ such that every irreducible component of $D$ and
every irreducible component of $D_{\sing}$ is totally invariant for $f^N$.
\end{proposition}

\begin{remark} \label{remark:dim}
The proof will use the following basic fact:
let $f\colon X\to X$ be a finite surjective morphism between complex manifolds, and let $Z \subset X$
be an irreducible subvariety. Then every irreducible component of $\fibre{f}{Z}$ has dimension equal to $\dim Z$. Indeed the morphism $f$ is flat and so is its base change $\holom{f|_{\fibre{f}{Z}}}{\fibre{f}{Z}}{Z}$. Thus we can apply \cite[III, Cor.9.6]{Har77}.
\end{remark}

\begin{proof}[Proof of Proposition \ref{prop:periodic-component}]
The map $f$ acts by permutation on the irreducible components of $D$, so it is clear
that up to replacing $f$ by some $f^N$ each irreducible component is totally invariant.

Let $d_1>\cdots>d_s$ be the dimensions of the irreducible components of $D_{\sing}$.
We argue by descending induction on these dimensions, replacing $f$ by a suitable iterate at each step.

{\em Start of the induction.}
Let ${\mathcal S}_{d_1}=\{Z_1,\dots,Z_k\}$ be the set of irreducible components of $D_{\sing}$ of dimension $d_1$. By Lemma \ref{lem:backward}
every irreducible component of $f^{-1}(Z_i)$ is contained in $D_{\sing}$. By Remark \ref{remark:dim} it has the maximal dimension $d_1$ and therefore determines an element of ${\mathcal S}_{d_1}$.
Since ${\mathcal S}_{d_1}$ has only finitely many elements this shows that
$f^{-1}(Z_i)$ is irreducible for every $i \in \{1, \ldots, k\}$ and
$f$ acts by permutation on ${\mathcal S}_{d_1}$.
Replacing $f$ by a suitable iterate, we may therefore assume that every element of ${\mathcal S}_{d_1}$ is totally invariant.

{\em Induction step.}
Suppose now that for some $\ell<s$  every element of ${\mathcal S}_{d_1}, \ldots, {\mathcal S}_{d_l}$ is totally invariant
for $f$ (up to replacing by an iterate).
Let ${\mathcal S}_{d_{\ell+1}}=\{Z_1,\dots,Z_k\}$ be the set of irreducible components of $D_{\sing}$ of dimension $d_{\ell+1}$.

Fix $i \in \{1, \ldots, k\}$, and let $W$ be an irreducible component of $f^{-1}(Z_i)$.
By Lemmas~\ref{lem:backward} and Remark \ref{remark:dim} the variety $W$ is contained in $D_{\sing}$ 
and has dimension $d_{\ell+1}$.
We claim that $W$ is itself an irreducible component of $D_{\sing}$.

{\em Proof of the claim.}
If $W$ is properly contained in an irreducible component $T$ of $D_{\sing}$, then necessarily $\dim T>d_{\ell+1}$, so $T$ is totally invariant by the induction hypothesis.
So $W\subseteq T = \fibre{f}{T}$  would imply
$$
Z_i=f(W)\subseteq f(T)=T.
$$
Yet $Z_i$ and $T$ are both irreducible components of $D_{\sing}$, so we get a contradiction.

By the claim every irreducible component of $f^{-1}(Z_i)$ belongs to ${\mathcal S}_{d_{\ell+1}}$.
Since ${\mathcal S}_{d_{\ell+1}}$ has only finitely many elements this shows that
$f^{-1}(Z_i)$ is irreducible for every $i \in \{1, \ldots, k\}$
and we can conclude as above.
\end{proof}

\section{Singularity theory}

The theory of Bernstein-Sato polynomials has a rich history with many deep contributions, e.g. by Kashiwara \cite{Kashiwara1976}, Malgrange \cite{Mal74} and Saito \cite{Saito1993}.
For our purpose we will use only rather elementary properties which we now recall following the easily accessible survey papers of Budur \cite{Budur2012} and Popa \cite{Popa2021}.

\begin{definition} \label{defn:bs-global}
Let $U$ be a complex manifold, and let $f\in \cO_{U}$ be a non-invertible holomorphic function.
The Bernstein--Sato polynomial $b_{f}(s)\in \CC[s]$ is the monic polynomial of minimal degree for which there exists
$P(s)\in \mathscr{D}_{U}[s]$ such that the relation 
\[
 P(s)\,f^{s+1}=b_{f}(s)\,f^s
\]
holds in the $\mathscr{D}_{U}$-module $\cO_{U}[f^{-1}, s]\cdot f^s$.
\end{definition}

\begin{remark} \label{remark:bs-global}
The polynomials $b(s)$ satisfying an identity as in Definition \ref{defn:bs-global} form an ideal
in $\C[s]$, the Bernstein-Sato polynomial is the unique monic generator of this principal ideal \cite[Defn.2.2.3]{Popa2021}.
\end{remark}

If we fix a point $x \in U$ we can always find a neighbourhood $U' \subset U$ such that 
$b_{f, U'}(s)$ divides $b_{f, V}(s)$ for every neighbourhood $x \in V \subset U'$ \cite[Lemma 2.2.4]{Popa2021}. Thus we can define:

\begin{definition} \label{defn:bs-local}
Let $U$ be a complex manifold, and let $f\in \cO_{U}$ be a holomorphic function.
Let $x \in U$ be a point such that $f(x)=0$. 
The local Bernstein-Sato polynomial of $f$ at $x$ is
$$
b_{f,x}(s) := b_{f, U'}(s)
$$
where $U' \subset U$ is a sufficiently small neighbourhood of $x$.
\end{definition}

It is classical that $s+1$ divides $b_{f,x}(s)$ and that all roots of $b_{f,x}(s)$ are negative rational numbers \cite{Kashiwara1976}. In special cases the Bernstein-Sato polynomials are well-known

\begin{example} \label{example:BSpolynomials} 
\begin{enumerate}
\item Assume that $f$ has a normal crossings singularity of rank two in $0 \in U \subset \C^n$, i.e., in suitable analytic coordinates we have $f(x)=x_1 x_2$. Then
$$
b_{f,0} = (s+1)^2.
$$
More generally if $f(x)=x_1^q x_2^q$ for some $q \in \N$, then by \cite[Ex.2.2.11]{Popa2021} 
$$
b_{f,0} = \prod_{k=0}^{q-1} (s+1-\frac{k}{q})^2.
$$
\item Assume that $f$ has a pinch point in $0 \in U \subset \C^n$, i.e., in suitable analytic coordinates we have $f(x)=x_1^2 + x_2^2 x_3$. Then by \cite[Section~5.4]{CDNMS2022}
$$
b_{f,0} = (s+1)^2\left(s+\frac{3}{2}\right).
$$ 
\end{enumerate}
\end{example}

\begin{definition} \label{defn:minimal-exponent}
Let $U$ be a complex manifold, and let $f\in \cO_{U}$ be a holomorphic function such that $f(x)=0$
and the divisor $\{ f=0\}$ is reduced in $x$.
The reduced Bernstein--Sato polynomial is
\[
 \widetilde b_{f,x}(s):=\frac{b_{f,x}(s)}{s+1}.
\]
The minimal exponent of $f$ at $x$ is
\[
 \widetilde\alpha_x(f):=-\max\{\rho\in\QQ\mid \rho \text{ is a root of }\widetilde b_{f,x}(s)\}
\]
with the convention $\widetilde\alpha_x(f)=\infty$ if $\widetilde b_{f,x}(s)=1$.
\end{definition}

We will later use the following easy lemma:

\begin{lemma} \label{lemma:special-open}
Let $U$ be an affine manifold, and let $f\in \cO_{U}$ be a holomorphic function
defining a reduced divisor $D \subset U$. Assume that there exists a point $x_0 \in D$ such that
$b_{f,x_0}(s)=(s+1)^d$ for some $d \in \N$. Then there exists an affine neighbourhood $x_0 \subset U' \subset U$
such that for all $x \in D \cap U'$ one has  $b_{f,x}(s)=(s+1)^{d_x}$ for some $d_x \in \N$.
\end{lemma}

\begin{proof}
By Definition \ref{defn:bs-local} there exists a neighbourhood $x_0 \subset U' \subset U$ such that 
$
b_{f, U'}(s) = b_{f,x_0}(s)
$.
Since $U'$ is affine we then know by  \cite[Defn.2.2.6]{Popa2021} that
$$
b_{f, U'}(s) = \mbox{lcm}_{x \in U'} b_{f,x}(s).
$$
Thus $b_{f,x}(s)$ divides $b_{f,x_0}(s) = (s+1)^d$.
\end{proof}

We recall some basic facts on Milnor fibres and the spectrum :

\begin{definition} \label{defn:milnor} \cite[Sect.1.1]{Budur2012}
Let 
\[
        f\colon(\C^r,0)\to(\C,0)
\]
be a holomorphic germ with a singularity at \(0\).  The Milnor fiber of $f$ at the origin is
\[
M_{f, 0} := B_\epsilon\cap f^{-1}(t).
\]
where $B_\epsilon$ is ball of radius $0<\eta \ll 1$ around the origin and $t \in \C$ is a point with $0<|t| \ll \eta$.
\end{definition}

The cohomology groups $H^i(M_{f,0},\C)$ admit an action $T$ called monodromy gener-
ated by going once around a loop starting at $t$ around $0$. This action can be encoded by different invariants

\begin{definition} \label{defn:spectrum} \cite[Sect.4.6]{DMST06}
Let 
\[
        f\colon(\C^r,0)\to(\C,0)
\]
be a polynomial germ with a singularity at \(0\). The Steenbrink spectrum of $f$
at $0$ is
\begin{equation*}
%\label{eqSpectrum}
\mbox{Sp}(f,0)=\sum_{c>0}n_{c,0}(f)\cdot t^c,
\end{equation*}
where  the spectrum multiplicities
$$
n_{c,0}(f):=\sum_{i\in\Z}(-1)^{r-1-i} \dim_\C \mbox{Gr}_F ^{\lfloor r-c \rfloor}\widetilde{H}^{i}(M_{f,0},\C)_{e^{-2\pi ic}}$$ 
record the generalized Euler
characteristic on the $\lfloor r-c \rfloor$-graded piece of the Hodge
filtration on the $\exp (-2\pi ic)$-monodromy eigenspace on the reduced cohomology of the Milnor fiber.
\end{definition}

\begin{remark} \label{remark:cohomology-isolated}
If $f$ has an isolated singularity at the origin, the formulas above can be considerably simplified:
by \cite[Theorem~6.5]{Milnor1968} the Milnor fibre has the homotopy type of a bouquet of \((r-1)\)-spheres and the numbers of spheres is given by the Milnor number $\mu(f)$. In particular 
the only non-zero cohomology groups are
$$
H^0(M_{f, 0} ,\C)\simeq \C, \qquad H^{r-1}(M_{f, 0} ,\C) \simeq \C^{\mu(f)},
$$
so the only non-zero reduced cohomology is 
$$
\wt H^{r-1}(M_{f, 0} ,\C)  \simeq \C^{\mu(f)}.
$$  
Thus we have
\begin{equation}
\label{multiplicity-easy}
n_{c,0}(f) =  \dim \mbox{Gr}_F ^{\lfloor r-c \rfloor}\widetilde{H}^{r-1}(M_{f,0},\C)_{e^{-2\pi ic}}
\end{equation}
\end{remark}

We will use A'Campo's Lefschetz-number theorem in the following form.

\begin{theorem}\label{thm:acampo}
Let $g\colon(\C^r,0)\to(\C,0)$ be a holomorphic germ. If
$g\in\mm_{\C^r,0}^2$, then the Lefschetz number of its local monodromy is
zero:
\[
 L(T_g):=\sum_i(-1)^i\Tr\bigl(T_g\mid H^i(M_{g,0},\C)\bigr)=0.
\]
\end{theorem}

This is \cite[Thm.~1]{ACampo1973}; %the more general vanishing-cycle statement is \cite[Th\'eor\`eme~1 bis]{ACampo1973}.

\begin{lemma}\label{lem:acampo-trace}
Let $g\colon(\C^r,0)\to(\C,0)$, with $r\geq2$, have an isolated critical
point. If every eigenvalue of the monodromy on
$\widetilde H^{r-1}(M_{g,0},\C)$ is equal to $1$, then $r$ is even and
$\mu(g)=1$. Moreover we can choose local coordinates such that
$g=z_1^2+\cdots+z_r^2$.
\end{lemma}

\begin{proof}
By Remark~\ref{remark:cohomology-isolated}, the Milnor fibre is connected
and its only nonzero reduced cohomology group is
$\widetilde H^{r-1}(M_{g,0},\C)$, of dimension $\mu(g)$. The eigenvalue
hypothesis gives
\[
 \Tr\bigl(T_g\mid\widetilde H^{r-1}(M_{g,0},\C)\bigr)=\mu(g).
\]
Since $g$ has a critical point,
$g\in\mm_{\C^r,0}^2$, and Theorem~\ref{thm:acampo} gives
\[
 0=L(T_g)=1+(-1)^{r-1}\mu(g).
\]
Thus $r$ is even and $\mu(g)=1$. The holomorphic Morse lemma
\cite[Theorem~I.2.46]{GreuelLossenShustin2007} gives the asserted normal
form.
\end{proof}

The following statement can be seen as a version of \cite[Thm.~2]{ACampo1973} for non-isolated singularities:

\begin{proposition} \label{prop:codim-one}
Let 
\[
        f\colon(\C^n,0)\to(\C,0)
\]
be a holomorphic germ defining a reduced divisor $D$ with a singularity at
\(0\). Assume that
$$
b_{f,0}(s) = (s+1)^d
$$
for some $d \geq 2$. Put
\[
 r:=\codim_{\C^n,0}D_{\sing}.
\]
If the singularity is isolated, set $H=\C^n$ and $r=n$. Otherwise assume
that there exists a smooth $r$-dimensional germ $H$ through the origin,
transverse to every stratum of
a Whitney regular stratification, such that $f|_H$ defines a reduced
divisor $D_H$ with an isolated singularity at \(0\).

Then the hypersurface $D$ is not normal, i.e. the singular locus
of $D$ has codimension one in the origin.
\end{proposition}

\begin{proof}
The hypersurface $D$ is not smooth at the origin.
Since $D$ is reduced and singular, its singular locus has codimension at
least two in $\C^n$, so $r\geq2$.
Choose local coordinates identifying $(H,0)$ with $(\C^r,0)$.

The unique root of the Bernstein--Sato polynomial is $-1$, so by the Kashiwara--Malgrange
theorem \cite[Theorem~6.3.5]{Bjork1993} the unique eigenvalue of the monodromy
action is $1$. Thus the spectrum multiplicity $n_{c,0}(f)$ is zero for
every $c \not\in \Z$.

We claim that the monodromy action on $M_{f|_H, 0}$ has $\lambda=1$ as its unique eigenvalue: if this is not the case there exists a $c \in \Q \setminus \Z$ such that  $\lambda := e^{-2\pi i c} \neq 1$
is an eigenvalue. Since $f|_H$ has an isolated singularity at $0$ we know by Remark \ref{remark:cohomology-isolated} that $H^{r-1}(M_{f|_H, 0} ,\C)$ is the unique non-zero reduced cohomology group, so we get $H^{r-1}(M_{f|_H, 0} ,\C)_\lambda \neq 0$. Choose now an integer
$m$ such that 
$\mbox{Gr}_F ^{\lfloor r-c-m \rfloor}\widetilde{H}^{r-1}(M_{f|_H, 0},\C)_\lambda \neq 0$,
then by \eqref{multiplicity-easy} we have $n_{c+m, 0}(f|_H) \neq 0$.
Since $H \subset \C^n$ is general we know by \cite[Cor.1.5]{DMST06} that the Steenbrink spectra
$\mbox{Sp}(f,0)$ and $\mbox{Sp}(f|_H,0)$ coincide up to sign.
Therefore the spectrum multiplicty $n_{c+m,0}(f)$ is not zero. Yet $c+m$ is not an integer, a contradiction.

Lemma~\ref{lem:acampo-trace} shows that $\mu(f|_H)=1$ and that there exist local analytic
coordinates $z_1, \ldots, z_r$ on $H$ such that
$$
f|_H = \sum_{i=1}^r z_i^2, 
$$
i.e. $f|_H$ defines a hypersurface with an ordinary double point. By  \cite[Example~2.2]{Walther2002}
we have
$$
b_{f|_H, 0}(s) = (s+1) (s+\frac{r}{2}),
$$
so the minimal exponent $\widetilde\alpha_0(f|_H)$ is equal to $\frac{r}{2}$ (cf. Definition \ref{defn:minimal-exponent}). Since
$b_{f,0}(s) = (s+1)^d$ with $d \geq 2$ we have $\widetilde\alpha_0(f)=1$. By \cite[Thm.E(1)]{MustataPopa2020} we have
$$
\frac{r}{2} = \widetilde\alpha_0(f|_H) \leq \widetilde\alpha_0(f) =1.
$$
Thus we obtain $r=2$ and therefore $\dim D_{\sing} = n-2 = \dim D-1$.
\end{proof}

\begin{remark}
The proof above gives also a more geometric information: the restriction of $f$
to a general plane $\C^2 \subset \C^n$ can be written as
$$
f|_H = z_1^2+z_2^2 = (z_1 +i z_2) (z_1-iz_2),
$$
so we obtain a normal crossings singularity of rank two. 
\end{remark}

\section{A divisibility lemma for Bernstein--Sato polynomials}

In this section we prove our key technical result on the roots of the Bernstein-Sato polynomials
at periodic point.

\begin{lemma}\label{lem:power-b}
Let $(U,x)$ be a germ of a complex manifold and $0\ne h\in \cO_{U,x}$. For every integer $m\ge 1$ one has a divisibility
\[
 b_{h^m,x}(s)\ \bigm|\ \prod_{j=0}^{m-1} b_{h,x}(ms+j)
 \]
in $\CC[s]$.
Equivalently, every root of $b_{h^m,x}(s)$ is of the form $\frac{\rho-j}{m}$ for some root $\rho$ of $b_{h,x}(s)$ and some $j\in\{0,\dots,m-1\}$.
\end{lemma}

\begin{proof}
Let $b_{h,x}(s)$ be the Bernstein--Sato polynomial of $h$ at $x$.
By definition, there exists $P(s)\in \mathscr{D}_{U,x}[s]$ such that
\[
 P(s)\,h^{s+1}=b_{h,x}(s)\,h^s.
\]
For each $j\in\{0,\dots,m-1\}$, substitute $s\mapsto ms+j$ and obtain
\begin{equation}\label{eq:shifted}
 P(ms+j)\,h^{ms+j+1}=b_{h,x}(ms+j)\,h^{ms+j}.
\end{equation}
Applying \eqref{eq:shifted} successively for $j=m-1,m-2,\dots,0$ yields
\[
 \Bigl(P(ms)\cdots P(ms+m-1)\Bigr)\,h^{ms+m}
 =
 \Bigl(\prod_{j=0}^{m-1} b_{h,x}(ms+j)\Bigr)\,h^{ms}.
\]
Since $h^{ms+m}=(h^m)^{s+1}$ and $h^{ms}=(h^m)^s$, the polynomial
\[
 B(s):=\prod_{j=0}^{m-1} b_{h,x}(ms+j)
\]
satisfies a Bernstein--Sato type equation for $h^m$ at $x$.
By the minimality of $b_{h^m,x}(s)$, we conclude that $b_{h^m,x}(s)$ divides $B(s)$.
\end{proof}

The behaviour of the Bernstein-Sato polynomial under birational morphisms play an important role in
the proof of Kashiwara's theorem \cite[Thm.5.1]{Kashiwara1976}. For finite morphisms we have the following application 
of a result of \`Alvarez Montaner,  Huneke, and N\'u\~nez-Betancourt:

\begin{lemma}\label{lem:cross-point}
Let
\[
 \varphi\colon (X,x)\longrightarrow (Y,y)
\]
be a finite surjective morphism of germs of complex manifolds, and let
$0\ne h\in\cO_{Y,y}$ satisfy $h(y)=0$. Then
\[
 b_{h,y}(s)\ \bigm|\ b_{h\circ\varphi,x}(s).
\]
\end{lemma}

\begin{proof}
We fix analytic neighbourhoods $x \in U$ and $x \in V$ such that we have a finite holomorphic map
$$
\holom{g:=\varphi|_U}{U}{V}
$$
Up to shrinking the neighbourhoods we can assume that we are in the setup of  
Definition \ref{defn:bs-local}, i.e. we have 
\begin{equation}
\label{help}
b_{h,y}(s) = b_{h,V}(s) \qquad \mbox{and} \qquad 
b_{h \circ \varphi,x}(s) = b_{h\circ g, U}(s)
\end{equation}
Since $g$ is flat and $U$ is smooth, the injective morphism of rings
$$
g^\# : \sO_V \rightarrow g_* \sO_U, \qquad f \ \mapsto \ f \circ g
$$
gives $g_* \sO_U$ the structure of a locally free $\sO_V$-module and by \cite[Prop.5.7]{KM98} the trace map
$$
\beta := \frac{1}{\deg g} \mbox{\rm Trace}_{U/V} :  g_* \sO_U \rightarrow \sO_V
$$
defines a splitting of the inclusion $g^\#$.

Now we follow the proof of \cite[Thm.3.14]{AlvarezMontanerHunekeNunezBetancourt2017}:
let  
$P(s)\in \mathscr{D}_{U}[s]$ be the differential operator such that the relation 
\begin{equation}
\label{relation-U}
 P(s)\, (h\circ g)^{s+1}=b_{h\circ g,U}(s)\,(h\circ g)^s
\end{equation}
holds. By \cite[Lemma 3.1]{AlvarezMontanerHunekeNunezBetancourt2017}
we have $(\beta \circ P)(s)\in \mathscr{D}_{V}[s]$.
Now recall that by assumption $h \circ g$ is the image of $h$ under the ring morphism $g^\#$,
so $\beta\left((h \circ g)^s\right)= \beta(g^\# (h^s))=h^s$. 
Thus applying $\beta$ to \eqref{relation-U} yields
$$
(\beta \circ P)(s)\, h^{s+1}=b_{h \circ g,U}(s) \, h^s,
$$
i.e., the polynomial $b_{h \circ g,U}(s)$ satisfies a Bernstein-Sato type equation for $h$.
By Remark \ref{remark:bs-global} this implies that
$b_{h,V}(s)$ divides $b_{h \circ g,U}(s)$. By \eqref{help} this proves the statement. 
\end{proof}

\begin{proposition}\label{prop:key}
Let $(X,x)$ be a germ of a complex manifold, and let
$$
\holom{\varphi}{(X,x)}{(X,x)}
$$
be a finite endomorphism. Let $0\ne h\in\cO_{X,x}$ such that
\[
 h\circ \varphi = u\cdot h^m
\]
where $m \geq 2$ and $u \in\cO_{X,x}^{\times}$.
Then every root of $b_{h,x}(s)$ is at least $-1$.
\end{proposition}

\begin{proof}
By Lemma~\ref{lem:cross-point}, applied with $Y=X$ and $y=x$, we have
\begin{equation}
\label{divisibility}
b_{h,x}(s) \ \bigm|\ b_{h\circ\varphi,x}(s)
 =b_{u \cdot h^m,x}(s).
\end{equation}
Multiplying with a unit does not change the Bernstein-Sato polynomial, so we have
$b_{u \cdot h^m,x}(s) = b_{h^m,x}(s)$. 
Therefore combining \eqref{divisibility}
with Lemma \ref{lem:power-b} yields
\begin{equation}\label{eq:bh-div-product}
 b_{h,x}(s)\ \bigm|\ \prod_{j=0}^{m-1} b_{h,x}(ms+j).
\end{equation}

Let ${\mathcal R}$ be the finite set of roots of $b_{h,x}(s)$ and 
let $\rho\in {\mathcal R}$ be its smallest element. Arguing by contradiction we assume that $\rho<-1$.
By \eqref{eq:bh-div-product} there exists a $j\in\{0,\dots,m-1\}$ such that
$b_{h,x}(m\rho+j)=0$. 
Thus
\[
 \rho':=m\rho+j
\]
is an element of $\mathcal R$. Since $m \geq 2$ we have
\[
 \rho'-\rho=(m-1)\rho+j\le (m-1)\rho+(m-1)=(m-1)(\rho+1)<0.
\]
Thus $\rho'<\rho$, a contradiction to the minimality of $\rho$.
\end{proof}

\section{Proof of the main results}

\begin{proof}[Proof of Theorem \ref{thm:main}]
We assume that $D$ is not smooth and denote by $Z$ an arbitrary irreducible component of $D_{\sing}$.
By Proposition \ref{prop:periodic-component} we can replace $f$ by a
suitable iterate such that every irreducible component of $D$ and of
$D_{\sing}$ is totally invariant; in particular, $f^{-1}(Z)=Z$.
By Lemma~\ref{lem:restrictpolarised}, the restriction $f|_Z\colon Z\to Z$ is a finite polarised endomorphism.
Fix an algebraic Whitney stratification of $(X,D)$. There is a nonempty
Zariski open subset $Z_\circ\subset Z$ contained in a dense smooth stratum
of $Z$ and disjoint from the other irreducible components of $D_{\sing}$.
By \cite[Thm.~5.1]{Fakhruddin2003}, periodic points of $f|_Z$ are Zariski
dense in $Z$. Choose a periodic point $x\in Z_\circ$ and replace $f$ by a
further iterate so that $f(x)=x$.

Let now  $h$ be a reduced local equation of $D$ at $x$. Since
$f^{-1}(D)=D$ set-theoretically and $f$ is polarised we know by Lemma~\ref{lem:m=q} that $f^*D=mD$
for some $m>1$.
Thus we have
\[
 h\circ f = u\cdot h^m
\]
for some unit $u\in \cO_{X,x}^\times$.
By Proposition \ref{prop:key} every root of $b_{h,x}(s)$ is at least
$-1$. 
Let $\rho_{\max}$ be the largest root, then by \cite[Thm.2.7.2]{Popa2021} we have
$$
-\rho_{\max} = \mbox{lct}_x (X,D)
$$
where $\mbox{lct}_x (X,D)$ is the log-canonical threshold of the pair $(X,D)$ in the point $x$.
By \cite[Cor.3.3]{BroustetHoring2014} the pair $(X,D)$ is log-canonical, so
we obtain $-\rho_{\max}=1$. Thus all the roots of $b_{h,x}(s)$ are equal to $-1$, i.e. we have
$$
b_{h,x}(s) = (s+1)^d
$$
for some $l \in \N$. By Lemma \ref{lemma:special-open}, applied to some affine neighbourhood of $x$, the set of points $x' \in Z$ where
$b_{h,x'}(s) = (s+1)^{d_{x'}}$ is Zariski open in $Z$. Thus
up to shrinking the Zariski open subset $Z_\circ \subset Z$ 
 there exists a $d_{gen} \in \N$
such that for every 
$x' \in Z_\circ$ we have
$$
b_{h,x'}(s) = (s+1)^{d_{gen}}. 
$$
Since $Z \subset D_{\sing}$ we have $d_{gen} \geq 2$, e.g. by \cite[Proposition~3.28]{AlvarezMontanerJeffriesNunezBetancourt2022}.

Let $H_1, \ldots, H_{\dim Z}$ be general hyperplane sections of $X$ and set $H:= H_1 \cap \ldots \cap H_{\dim Z}$. 
By Bertini's theorem the linear space $H$ is transversal to every stratum of the Whitney stratification. 
Moreover the intersection
$H \cap Z_\circ$ is finite and not empty, so for a point $x_0 \in H \cap Z_\circ$ the linear space $H$
satisfies the conditions of  Proposition \ref{prop:codim-one}.
Thus we know that $D_{\sing} \subset D$ has codimension one in $x_0 \in Z_\circ$.
Since $Z_\circ$ is disjoint from the other irreducible components of $D_{\sing}$
this shows that $Z$
has codimension one in $D$.
\end{proof}

\begin{remark}
The proof above shows that in the situation of Theorem \ref{thm:main} every 
irreducible component $Z$ of $D_{\sing}$ has
codimension two in $X$. Since the pair $(X, D)$ is log-canonical, this shows that
the hypersurface $D$ has two smooth transverse branches in a general point of $Z$.
\end{remark}

\begin{proof}[Proof of Corollary \ref{cor:intro-pn}]
By assumption $D$ is normal, so by Theorem \ref{thm:main} the divisor $D$ is even smooth.
A smooth totally invariant irreducible hypersurface in $\PP^n$ must have degree one by \cite[Thm.~1 and Thm.~2]{CerveauLinsNeto2000} together with \cite[Prop.~8]{ParanjapeSrinivas1989}.
\end{proof}

\section{Invariant cubics and linearity in dimension four}
\label{sec:cubics-linearity}

We start with a special case:

\begin{proposition}\label{prop:F5-not-totally-invariant}
Let $f\colon\PP^4\to\PP^4$ be an endomorphism of the projective space,
and let
\[
 D_5=\bigl\{X_0^2X_2+X_1^3+X_1^2X_3+X_0X_1X_4=0\bigr\}
 \subset \PP^4
\]
be the cubic $F_5$ of \cite[Theorem~3.1(a.5)]{LPS11}. Then $D_5$ is not
totally invariant under $f$.
\end{proposition}

\begin{proof}
We argue by contradiction, so suppose that $D:=D_5$ is totally invariant under $f$.
Making a coordinate change 
$$
Y_i=X_i \quad \mbox{for} \quad i=\{0,1,2\}, \qquad Y_3=X_4, \qquad Y_4=X_1+X_3
$$
the equation of $D$ becomes
\[
 S:= Y_0^2Y_2+Y_0Y_1Y_3+Y_1^2Y_4.
\]
In these cooordinates the singular locus of $D$ is
\[
 Z :=D_{\sing}=\{Y_0=Y_1=0\}\simeq\PP^2.
\]
Up to replacing $f$ by an iterate, Proposition~\ref{prop:periodic-component}
yields
$
f^{-1}(Z)=Z,
$
and by Lemma~\ref{lem:restrictpolarised} the restriction
$$
\holom{g:=f|_Z}{Z \simeq \PP^2}{Z \simeq \PP^2} 
$$ 
is an endomorphism of $\PP^2$.

On the affine chart $Y_2 \neq 0$, we can rewrite the equation $S$ as

$$
\left(\frac{Y_0}{Y_2}\right)^2 + \frac{Y_0}{Y_2}
\frac{Y_1}{Y_2}\frac{Y_3}{Y_2} 
+ \left(\frac{Y_1}{Y_2}\right)^2 \frac{Y_4}{Y_2}
\\
=
\left(
\frac{Y_0}{Y_2} + \frac{1}{2} \frac{Y_1}{Y_2} \frac{Y_3}{Y_2} 
\right)^2
+
\left[
\frac{Y_4}{Y_2} - \frac{1}{4} \left(\frac{Y_3}{Y_2}\right)^2
\right]
\left(\frac{Y_1}{Y_2}\right)^2.
$$

Thus in the affine coordinates
$$
u = \frac{Y_0}{Y_2} + \frac{1}{2} \frac{Y_1}{Y_2} \frac{Y_3}{Y_2}, \
v= \frac{Y_3}{Y_2}, \
w = \frac{Y_4}{Y_2} - \frac{1}{4} \left(\frac{Y_3}{Y_2}\right)^2, \
y =  \frac{Y_1}{Y_2}
$$
the equation of $D$ simplifies to
\begin{equation}
\label{supersimple} u^2+wy^2.
\end{equation}
Set now
$$
\Gamma := \{ Y_0=Y_1=0, Y_4 Y_2 - \frac{1}{4} Y_3^2=0 \}.
$$
Then $\Gamma$ is a smooth conic in $D_{\sing} \simeq Z$ and its equations in the affine
chart $Y_2 \neq 0$ are $u=y=w=0$. Thus \eqref{supersimple} (and its analogues in the other charts) shows that
$D$ has a pinch point (cf. Example \ref{example:BSpolynomials}) along the curve $\Gamma$
and normal crossing singularities of rank two on $Z \setminus \Gamma$.

We claim that $\Gamma$ is totally invariant for the endomorphism $g$. Since $\Gamma$ is a smooth conic 
this contradicts \cite[Thm.2]{CerveauLinsNeto2000} and finishes the proof.

{\em Proof of the claim.} Arguing by contradiction we assume that there
exist points $z \in \Gamma$ and $p \in (Z \setminus \Gamma)$ such that $g(p)=z$. 
Since $D$ is totally invariant for $f$ we have $f^*D = q D$ by Lemma \ref{lem:m=q}.
In terms of the equation $S$ this means that 
$S \circ f = c S^q$ where $c \neq 0$ is a unit. 
Applying
Lemma~\ref{lem:cross-point} to the germ
\[
 f\colon(\PP^4,p)\longrightarrow(\PP^4,z)
\]
gives
\begin{equation}
\label{divide-cross}  b_{S,z}(s)\mid b_{c S^q,p}(s).
\end{equation}
Since $D$ has normal crossing singularities of rank two in $p$, we know by Example \ref{example:BSpolynomials} that
$$
b_{c S^q,p}(s) = \prod_{k=0}^{q-1} (s+1-\frac{k}{q})^2.
$$
Yet $D$ has a pinch point in $Z$, so again by Example \ref{example:BSpolynomials} 
$$
b_{S,z}(s) =  (s+1)^2\left(s+\dfrac{3}{2}\right).
$$
Thus $-\frac{3}{2}$ is a root of $b_{S,z}(s)$, but not of $b_{c S^q,p}(s)$.
This contradicts \eqref{divide-cross}.
\end{proof}

\begin{proof}[Proof of Corollary~\ref{cor:cubic}]
We argue by contradiction and assume that there exists an endomorphism
$f\colon\PP^n\to\PP^n$ and a prime divisor $D\subset\PP^n$ of degree three that is totally invariant.
We can assume that we have chosen an endomorphism $f$ realising the smallest possible
value for $n \in \N$.  Since Conjecture \ref{conj:linear} is known for $n \leq 3$ we have $n \geq 4$.  
By \cite[Theorem~1.3]{Mabed2023} the minimality implies that $D$ is not a cone.
By Corollary \ref{cor:intro-pn} we can assume that $D$ is not normal, so
we can a apply \cite[Thm.3.11]{LPS11} to obtain that $n=4$ and $D$ is the cubic threefold $F_5$.
Yet this contradicts Proposition~\ref{prop:F5-not-totally-invariant}.
\end{proof}

\begin{proof}[Proof of Corollary~\ref{cor:linear-p4}]
Let $D\subset\PP^4$ be a totally invariant prime divisor of degree $d$.
By \cite[Theorem~2.1]{HN11} we have $d \leq 4$ and the case $d=4$ is
excluded by \cite[Theorem~1.1]{Hoe17}. The case $d=3$ is excluded by
Corollary~\ref{cor:cubic}.
A prime divisor of degree two is normal, so this case is excluded by
Corollary~\ref{cor:intro-pn}. Thus $D$ is a hyperplane. 
\end{proof}

%\bibliographystyle{amsalpha}
%\bibliography{biblio}

\begin{thebibliography}{AMHNnB17}

\bibitem[A'C73]{ACampo1973}
Norbert A'Campo, \emph{Le nombre de {Lefschetz} d'une monodromie}, Nederl.
  Akad. Wet., Proc., Ser. A \textbf{76} (1973), 113--118 (French).

\bibitem[AMHNnB17]{AlvarezMontanerHunekeNunezBetancourt2017}
Josep \`Alvarez~Montaner, Craig Huneke, and Luis N\'u\~nez Betancourt,
  \emph{D-modules, {B}ernstein--{S}ato polynomials and {F}-invariants of direct
  summands}, Advances in Mathematics \textbf{321} (2017), 298--325.

\bibitem[AMJNnB22]{AlvarezMontanerJeffriesNunezBetancourt2022}
Josep \`Alvarez~Montaner, Jack Jeffries, and Luis N\'u\~nez Betancourt,
  \emph{Bernstein--{S}ato polynomials in commutative algebra}, Commutative
  Algebra: Expository Papers Dedicated to David Eisenbud on the Occasion of His
  75th Birthday, Springer, Cham, 2022, pp.~1--76.

\bibitem[BD01]{BD01}
Jean-Yves Briend and Julien Duval, \emph{Two characterizations of equilibrium
  measure of an endomorphism of {{\(P^k(\mathbb{C})\)}}}, Publ. Math., Inst.
  Hautes {\'E}tud. Sci. \textbf{93} (2001), 145--159 (French).

\bibitem[Bea01]{Beauville2001}
Arnaud Beauville, \emph{Endomorphisms of hypersurfaces and other manifolds},
  International Mathematics Research Notices \textbf{2001} (2001), no.~1,
  53--58.

\bibitem[BG17]{BG17}
Ama{\"e}l Broustet and Yoshinori Gongyo, \emph{Remarks on log {Calabi}-{Yau}
  structure of varieties admitting polarized endomorphisms}, Taiwanese J. Math.
  \textbf{21} (2017), no.~3, 569--582 (English).

\bibitem[BH14]{BroustetHoring2014}
Ama\"el Broustet and Andreas H\"oring, \emph{Singularities of varieties
  admitting an endomorphism}, Mathematische Annalen \textbf{360} (2014),
  no.~1-2, 439--456.

\bibitem[Bj{\"o}93]{Bjork1993}
Jan-Erik Bj{\"o}rk, \emph{Analytic {$D$}-modules and applications}, Math.
  Appl., Dordr., vol. 247, Dordrecht: Kluwer Academic Publishers, 1993
  (English).

\bibitem[Bud12]{Budur2012}
Nero Budur, \emph{Singularity invariants related to {Milnor} fibers: survey},
  Zeta functions in algebra and geometry. Second international workshop,
  Universitat de les Illes Balears, Palma de Mallorca, Spain, May 3--7, 2010.
  Proceedings, Providence, RI: American Mathematical Society (AMS); Madrid:
  Real Sociedad Matem{\'a}tica Espa{\~n}ola, 2012, pp.~161--187 (English).

\bibitem[CDNMS22]{CDNMS2022}
Alberto Casta{\~n}o~Dom{\'i}nguez, Luis Narv{\'a}ez~Macarro, and Christian
  Sevenheck, \emph{Hodge ideals of free divisors}, Selecta Math. (N.S.)
  \textbf{28} (2022), no.~3, Paper No. 57, 62.

\bibitem[CL26]{CL26}
Wentao Chang and Yujie Luo, \emph{Sharp {Bounds} for {Totally} {Invariant}
  {Cycles} of {Projective} {Varieties}}, Preprint, {arXiv}:2607.27738
  [math.{AG}] (2026), 2026.

\bibitem[CLN00]{CerveauLinsNeto2000}
Dominique Cerveau and Alcides Lins~Neto, \emph{Hypersurfaces exceptionnelles
  des endomorphismes de {{$\mathbb{C}P(n)$}}}, Boletim da Sociedade Brasileira
  de Matem\'atica. Nova S\'erie \textbf{31} (2000), no.~2, 155--161.

\bibitem[CZ25]{CZ25}
Wentao Chang and De-Qi Zhang, \emph{Log {Calabi}-{Yau} structure of algebaic
  varieties admitting a polarized endomorphism}, Preprint, {arXiv}:2509.17927
  [math.{AG}] (2025), 2025.

\bibitem[DMST06]{DMST06}
A.~Dimca, Ph. Maisonobe, M.~Saito, and T.~Torrelli, \emph{Multiplier ideals,
  {$V$}-filtrations and transversal sections}, Math. Ann. \textbf{336} (2006),
  no.~4, 901--924 (English).

\bibitem[Fak03]{Fakhruddin2003}
Najmuddin Fakhruddin, \emph{Questions on self maps of algebraic varieties},
  Journal of the Ramanujan Mathematical Society \textbf{18} (2003), no.~2,
  109--122.

\bibitem[FS94]{FS94}
John~Erik Fornaess and Nessim Sibony, \emph{Complex dynamics in higher
  dimension. {I}}, Complex analytic methods in dynamical systems. Proceedings
  of the congress held at Instituto de Matem\'atica Pura e Aplicada, IMPA, Rio
  de Janeiro, Brazil, January 1992, Paris: Soci{\'e}t{\'e} Math{\'e}matique de
  France, 1994, pp.~201--231 (English).

\bibitem[GLS07]{GreuelLossenShustin2007}
Gert-Martin Greuel, Christoph Lossen, and Eugenii Shustin, \emph{Introduction
  to singularities and deformations}, Springer Monographs in Mathematics,
  Springer, Berlin, 2007.

\bibitem[GR20]{GiraldoRoeder2020}
Luis Giraldo and Roland K.~W. Roeder, \emph{Pulling back singularities of
  codimension one objects}, Proceedings of the American Mathematical Society
  \textbf{148} (2020), no.~3, 1207--1217.

\bibitem[Har77]{Har77}
Robin Hartshorne, \emph{Algebraic geometry}, Springer-Verlag, New York, 1977,
  Graduate Texts in Mathematics, No. 52. \MR{MR0463157 (57 \#3116)}

\bibitem[HN11]{HN11}
Jun-Muk Hwang and Noboru Nakayama, \emph{On endomorphisms of {F}ano manifolds
  of {P}icard number one}, Pure Appl. Math. Q. \textbf{7} (2011), no.~4,
  Special Issue: In memory of Eckart Viehweg, 1407--1426. \MR{2918167}

\bibitem[H{\"o}r17]{Hoe17}
Andreas H{\"o}ring, \emph{Totally invariant divisors of endomorphisms of
  projective spaces}, Manuscr. Math. \textbf{153} (2017), no.~1-2, 173--182
  (English).

\bibitem[Kas76]{Kashiwara1976}
Masaki Kashiwara, \emph{B-functions and holonomic systems. {Rationality} of
  roots of {B}-functions}, Invent. Math. \textbf{38} (1976), 33--53 (English).

\bibitem[KM98]{KM98}
J{\'a}nos Koll{\'a}r and Shigefumi Mori, \emph{Birational geometry of algebraic
  varieties}, Cambridge Tracts in Mathematics, vol. 134, Cambridge University
  Press, Cambridge, 1998, With the collaboration of C. H. Clemens and A. Corti.
  \MR{MR1658959 (2000b:14018)}

\bibitem[LM26]{LM26}
Yujie Luo and Sheng Meng, \emph{Log {Calabi}--{Yau} structure for endomorphisms
  on {$\mathbf{P}^n$}}, Preprint, {arXiv}:2608.02114 [math.{AG}] (2026), 2026.

\bibitem[LPS11]{LPS11}
Wanseok Lee, Euisung Park, and Peter Schenzel, \emph{On the classification of
  non-normal cubic hypersurfaces}, J. Pure Appl. Algebra \textbf{215} (2011),
  no.~8, 2034--2042 (English).

\bibitem[Mab23]{Mabed2023}
Yanis Mabed, \emph{Totally invariant divisors of non trivial endomorphisms of
  the projective space}, Geometriae Dedicata \textbf{217} (2023), 79.

\bibitem[Mal74]{Mal74}
Bernard Malgrange, \emph{Int{\'e}grales asymptotiques et monodromie}, Ann. Sci.
  {\'E}c. Norm. Sup{\'e}r. (4) \textbf{7} (1974), 405--430 (French).

\bibitem[Men23]{Men23}
Sheng Meng, \emph{Log {Calabi}-{Yau} structure of projective threefolds
  admitting polarized endomorphisms}, Int. Math. Res. Not. \textbf{2023}
  (2023), no.~24, 21272--21289 (English).

\bibitem[Mil68]{Milnor1968}
John~W. Milnor, \emph{Singular points of complex hypersurfaces}, Ann. Math.
  Stud., vol.~61, Princeton University Press, Princeton, NJ, 1968 (English).

\bibitem[MP20]{MustataPopa2020}
Mircea Musta\c{t}\u{a} and Mihnea Popa, \emph{Hodge ideals for {Q}-divisors,
  {$V$}-filtration, and minimal exponent}, Forum of Mathematics, Sigma
  \textbf{8} (2020), e19.

\bibitem[MYY25]{MYY25}
Joaqu{\'{\i}}n Moraga, Wern Yeong, and Jos{\'e}~Ignacio Y{\'a}{\~n}ez,
  \emph{Polarized endomorphisms of {Fano} varieties with complements}, Int.
  Math. Res. Not. \textbf{2025} (2025), no.~8, 18 (English), Id/No rnaf093.

\bibitem[Nak23]{Nak23}
Noboru Nakayama, \emph{Singularity of normal complex analytic surfaces
  admitting non-isomorphic finite surjective endomorphisms}, Osaka J. Math.
  \textbf{60} (2023), no.~2, 403--489 (English).

\bibitem[Pop21]{Popa2021}
Mihnea Popa, \emph{{$\mathcal{D}$}-modules in birational geometry}, Expanded
  lecture notes for Math 296, Harvard University, 2021.

\bibitem[PS89]{ParanjapeSrinivas1989}
K.~H. Paranjape and V.~Srinivas, \emph{Self-maps of homogeneous spaces},
  Invent. Math. \textbf{98} (1989), no.~2, 425--444. \MR{1016272 (91a:14034)}

\bibitem[Sai93]{Saito1993}
Morihiko Saito, \emph{On {$b$}-function, spectrum and rational singularity},
  Mathematische Annalen \textbf{295} (1993), no.~1, 51--74.

\bibitem[Wal02]{Walther2002}
Uli Walther, \emph{{$\mathcal{D}$}-modules and cohomology of varieties},
  Computations in algebraic geometry with Macaulay 2, Berlin: Springer, 2002,
  pp.~281--323 (English).

\bibitem[Yos22]{Yos22}
Shou Yoshikawa, \emph{Singularities of non-{{\(\mathbb{Q}\)}}-{Gorenstein}
  varieties admitting a polarized endomorphism}, Int. Math. Res. Not.
  \textbf{2022} (2022), no.~13, 10095--10118 (English).

\bibitem[Zha14]{Zha14}
De-Qi Zhang, \emph{Invariant hypersurfaces of endomorphisms of projective
  varieties}, Adv. Math. \textbf{252} (2014), 185--203. \MR{3144228}

\bibitem[Zho21]{Zho21}
Guolei Zhong, \emph{Totally invariant divisors of int-amplified endomorphisms
  of normal projective varieties}, J. Geom. Anal. \textbf{31} (2021), no.~3,
  2568--2593 (English).

\end{thebibliography}

\providecommand{\bysame}{\leavevmode\hbox to3em{\hrulefill}\thinspace}
\providecommand{\MR}{\relax\ifhmode\unskip\space\fi MR }
\providecommand{\MRhref}[2]{%
  \href{http://www.ams.org/mathscinet-getitem?mr=#1}{#2}
}
\providecommand{\href}[2]{#2}

\end{document}